\documentclass{amsart}

\usepackage[latin1]{inputenc}
\usepackage{amsfonts}
\usepackage{eufrak}
\usepackage{amsmath}
\usepackage{amsthm}
\usepackage{a4wide}
\usepackage[all]{xy}
\usepackage{textcomp}

\newtheorem{thm}{Theorem}[section]
\newtheorem{prop}[thm]{Proposition}
\newtheorem{rmk}[thm]{Remark}
\newtheorem{lemma}[thm]{Lemma}

\begin{document}

\title[A bi-invariant metric on the group of contactomorphisms of $\mathbb{R}^{2n}\times S^{1}$]
{An integer valued bi-invariant metric on the group of contactomorphisms of $\mathbb{R}^{2n}\times S^{1}$}

\author{Sheila Sandon}
\address{Departamento de Matem\'{a}tica, Instituto Superior T\'{e}cnico, Av. Rovisco Pais, 1049-001 Lisboa, Portugal}
\email{sandon@math.ist.utl.pt}

\begin{abstract}
\noindent
In his article \cite{V} on generating functions Viterbo constructed a bi-invariant metric on the group of compactly supported Hamiltonian symplectomorphisms of $\mathbb{R}^{2n}$. Using the set-up of \cite{mio} we extend the Viterbo metric to the group $\text{Cont}_0^{\phantom{0}c}\,(\mathbb{R}^{2n}\times S^1)$ of compactly supported contactomorphisms of $\mathbb{R}^{2n}\times S^{1}$ isotopic to the identity. We also prove that $\text{Cont}_0^{\phantom{0}c}\,(\mathbb{R}^{2n}\times S^1)$ is unbounded with respect to this metric.
\end{abstract}

\maketitle

\section{Introduction}

Gromov's non-squeezing theorem \cite{Gr} marked the beginning of the study of symplectic capacities, which are global symplectic invariants that measure the size of symplectic manifolds. Some years later Hofer \cite{H} discovered a way of measuring the size (or \textit{energy}) of Hamiltonian symplectomorphisms, by looking at the total variation of their generating Hamiltonian. This notion gave rise in fact to the definition of a bi-invariant metric (the Hofer metric) on the group of Hamiltonian symplectomorphisms \cite{H,P,LM}. The Hofer metric is deeply related to symplectic capacities. It was proved by Hofer that for any domain $\mathcal{U}$ of $\mathbb{R}^{2n}$ the Hofer-Zehnder capacity $c_{HZ}(\mathcal{U})$ is a lower bound for the energy of any compactly supported Hamiltonian symplectomorphism $\phi$ such that $\phi(\mathcal{U})\cap \mathcal{U}=\emptyset$. Lalonde and McDuff \cite{LM} gave moreover a direct geometric construction relating non-degeneracy of the Hofer metric to Gromov's non-squeezing theorem.\\
\\
Using the theory of generating functions, Viterbo constructed in \cite{V} a new symplectic capacity for domains of $\mathbb{R}^{2n}$ and a new bi-invariant metric $d_V$ on the group $\text{Ham}^c\,(\mathbb{R}^{2n})$ of compactly supported Hamiltonian symplectomorphisms\footnote{It was proved by Bialy and Polterovich \cite{BP} that the Viterbo and the Hofer metric coincide on a neighborhood of the identity in $\text{Ham}^c\,(\mathbb{R}^{2n})$, but Sorrentino and Viterbo \cite{SV} recently found examples showing that the two metrics are different in general.}, and proved an energy-capacity inequality relating these two notions. Moreover he defined a partial order $\leq_V$ on $\text{Ham}^c\,(\mathbb{R}^{2n})$, giving $\big(\,\text{Ham}^c\,(\mathbb{R}^{2n}),d_V\,\big)$ the structure of a partially ordered metric space \footnote{Recall that a \emph{partially ordered metric space} is a metric space $(Z,d)$ endowed with a partial order $\leq$ such that for every $a$, $b$, $c$ in $Z$ with $a\leq b\leq c$ it holds $d(a,b)\leq d(a,c)$.}.\\
\\
The generalization of the constructions in \cite{V} to the contact case \cite{B,mio} has been motivated by the theory of \textit{orderability} of contact manifolds introduced by Eliashberg and Polterovich \cite{EP}, and by the related contact rigidity phenomena studied in \cite{EKP}. Bhupal \cite{B} extended the Viterbo partial order to the group $\text{Cont}_0^{\phantom{0}c}\,(\mathbb{R}^{2n+1})$ of compactly supported and isotopic to the identity contactomorphisms of $\mathbb{R}^{2n+1}$, thereby proving orderability of $\mathbb{R}^{2n+1}$. In \cite{mio} we obtained a new proof of the contact non-squeezing theorem of Eliashberg, Kim and Polterovich \cite{EKP} by extending the Viterbo capacity to the contact manifold $\mathbb{R}^{2n}\times S^1$.\\
\\
In introducing the notion of orderability Eliashberg and Polterovich were in fact motivated by the question of finding some geometric structure on the group of contactomorphisms. They showed in \cite{EP} that the concept of \textit{relative growth}, which is available in any partially ordered group, can be applied to the contactomorphism group of an orderable contact manifold $(M,\xi)$ in order to associate to it a metric space $\big(Z(M,\xi),\delta\big)$. This can be done by defining, in terms of the relative growth, a pseudo-distance $\delta$ on the group of those contactomorphisms of $(M,\xi)$ that are generated by a positive Hamiltonian, and then by considering the quotient of this group by the equivalence classes of elements which are at zero distance from each other.\\
\\
In the present article we show that in the case of the contact manifold $\mathbb{R}^{2n}\times S^{1}$ it is possible to define a bi-invariant metric directly on the group of all (not necessarily positive) compactly supported contactomorphisms isotopic to the identity. This metric is a generalization of the Viterbo metric and can be easily constructed by using the set-up developed in \cite{mio}. However a crucial difference is that, in contrast with the symplectic case, our metric only takes values in $\mathbb{Z}$. We refer to the introduction of \cite{mio} for an explanation, in terms of generating functions, of the special role played by the integers in the study of rigidity phenomena for the contact manifold $\mathbb{R}^{2n}\times S^{1}$.\\
\\
Our results can be summarized in the following theorem.

\begin{thm}\label{main}
The group $\text{Cont}_0^{\phantom{0}c}\,(\mathbb{R}^{2n}\times S^1)$ admits an unbounded integer valued bi-invariant metric $d$. This metric is compatible with the Bhupal partial order $\leq_B$ in the sense that $\leq_B$ turns $\big(\,\text{Cont}_0^{\phantom{0}c}\,(\mathbb{R}^{2n}\times S^1),d\,\big)$ into a partially ordered metric space.
\end{thm}

The proof of Theorem \ref{main} is based on results of \cite{mio}.\\
\\
This article is organized as follows. In the first two sections we recall the set-up of \cite{mio}. In particular, in Section \ref{prel} we give some preliminaries on generating functions and in Section \ref{inv} we discuss the generalization to $\text{Cont}_0^{\phantom{0}c}\,(\mathbb{R}^{2n}\times S^1)$ of the invariants $c^+$ and $c^-$ that were constructed by Viterbo in \cite{V}. In Section \ref{metric} we define the metric $d$ on $\text{Cont}_0^{\phantom{0}c}\,(\mathbb{R}^{2n}\times S^1)$, and discuss an energy-capacity inequality relating it to the contact capacity for domains of $\mathbb{R}^{2n}\times S^1$ that was constructed in \cite{mio}. In Section \ref{order} we recall the definition of the Bhupal partial order $\leq_B$ on $\text{Cont}_0^{\phantom{0}c}\,(\mathbb{R}^{2n}\times S^1)$ and show that $d$ and $\leq_B$ are compatible. We also show how this implies that the energy (i.e. the distance to the identity) does not decrease along contact isotopies that are generated by a non-negative Hamiltonian. In the last section we prove that $d$ is unbounded.

\subsection*{Acknowledgements}

I thank my supervisor Miguel Abreu for his support and mathematical guidance, and Leonid Polterovich for feedback on preliminary versions of this article. My research was supported by an FCT graduate fellowship, program POCTI-Research Units Pluriannual Funding Program through the Center for Mathematical Analysis Geometry and Dynamical Systems and Portugal/Spain cooperation grant FCT/CSIC-14/CSIC/08.

\section{Preliminaries on generating functions}\label{prel}

We start by presenting some preliminaries on generating functions, referring to \cite{mio} and the bibliography therein for more details and background information.\\
\\
Let $B$ be a closed manifold, and consider a function $S:E\rightarrow\mathbb{R}$ defined on the total space of a fiber bundle $p:E \longrightarrow B$. We will assume that $dS: E\longrightarrow T^{\ast}E$ is transverse to $N_E:=\{\,(e,\eta)\in T^{\ast}E \;|\; \eta = 0 \;\text{on} \;\text{ker}\,dp\,(e)\,\}$, so that the set $\Sigma_S$ of fiber critical points is a submanifold of $E$, of dimension equal to the dimension of $B$. To any $e$ in $\Sigma_S$ we associate an element $v^{\ast}(e)$ of $T^{\phantom{p}\ast}_{p(e)}B$ by defining $v^{\ast}(e)\,(X):=dS\,(\widehat{X})$ for $X \in T_{p(e)}B$, where $\widehat{X}$ is any vector in $T_eE$ with $p_{\ast}(\widehat{X})=X$. Then $i_S:\Sigma_S\longrightarrow T^{\ast}B$, $e\mapsto \big(p(e),v^{\ast}(e)\big)$ is an exact Lagrangian immersion, with $i_S^{\phantom{S}\ast}\, \lambda_{\text{can}}=d\,(S_{|\Sigma_S})$. Its lift to $\big(J^1B=T^{\ast}B\times\mathbb{R}\,,\,\text{ker}(dz-\lambda_{\text{can}})\big)$ is the Legendrian immersion $j_S:\Sigma_S\rightarrow J^1B$, $e\mapsto\big(p(e), v^{\ast}(e),S(e)\big)$. The function $S:E\rightarrow\mathbb{R}$ is called a \textit{generating function} for $L_S:=i_S\,(\Sigma_S)\subset T^{\ast}B$ and for its lift $\widetilde{L_S}:=j_S(\Sigma_S)\subset J^1B$. A fundamental property of $S$ is that its critical points correspond under $i_S$ to intersections of $L_S$ with the 0-section $0_B$, and under $j_S$ to  intersections of $\widetilde{L_S}$ with the 0-wall $0_B \times \mathbb{R}$. Note also that if $e$ is a critical point of $S$ then its critical value is given by the $\mathbb{R}$-coordinate of the point $j_S(e)$.\\
\\
A generating function $S:E \longrightarrow \mathbb{R}$ is said to be \textit{quadratic at infinity} if $p:E \longrightarrow B $ is a vector bundle and if there exists a non-degenerate quadratic form $Q_{\infty}: E \longrightarrow \mathbb{R}$ such that $dS-\partial_vQ_{\infty}: E \longrightarrow E^{\ast}$ is bounded, where $\partial_v$ denotes the fiber derivative. Existence of generating functions quadratic at infinity for all Legendrian submanifolds of $J^1B$ contact isotopic to the 0-section was proved by Chaperon \cite{Ch} and Th\'{e}ret \cite{Th}, and independently by Chekanov \cite{C}. Their theorem is a generalization of the analogous result for Lagrangian submanifolds of $T^{\ast}B$ Hamiltonian isotopic to the 0-section, that was proved by Sikorav \cite{S,S2} using ideas of \cite{LS} and \cite{Ch1}. A second fundamental result is the uniqueness theorem for generating functions quadratic at infinity, that is due to Viterbo \cite{V} and Th\'{e}ret \cite{Th,Th2}.\\
\\
Relying on the Uniqueness Theorem, Viterbo \cite{V} applied Morse theoretical methods to generating functions in order to define invariants for Lagrangian submanifolds of $T^{\ast}B$ Hamiltonian isotopic to the 0-section. As observed by Bhupal \cite{B}, Viterbo's invariants can also be defined in the more general class $\mathcal{L}$ of Legendrian submanifolds of $J^1B$ contact isotopic to the 0-section. The construction goes as follows. Let $L$ be an element of $\mathcal{L}$ with generating function $S:E\rightarrow \mathbb{R}$. Denote by $E^{a}$, for $a\in \mathbb{R}\cup\infty$, the sublevel set of $S$ at $a$, and by $E^{-\infty}$ the set $E^{-a}$ for $a>0$ big. We consider the inclusion $i_a: (E^a,E^{-\infty})\hookrightarrow (E,E^{-\infty})$, and the induced map on cohomology
$$ i_a^{\phantom{a}\ast}: H^{\ast}(B)\equiv H^{\ast}(E,E^{-\infty})\longrightarrow H^{\ast}(E^a,E^{-\infty})$$
where $H^{\ast}(B)$ is identified with $H^{\ast}(E,E^{-\infty})$ via the Thom isomorphism. For any $u\neq 0$ in $H^{\ast}(B)$ we define
$$ c(u,L)=c(u,S)=\text{inf}\,\{\,a\in \mathbb{R} \;|\;i_a^{\phantom{a}\ast}(u)\neq 0\,\}.$$
Note that $c(u,L)$ is a critical value of $S$. 

\begin{lemma}[\cite{V}]\label{inv_leg}
Let $\mu\in H^n(B)$ denote the orientation class of $B$, and $0_B$ the 0-section in $J^1B$. The map $H^{\ast}(B)\times \mathcal{L}\longrightarrow\mathbb{R}$, $(u,L)\longmapsto c(u,L)$ satisfies the following properties:
\vspace{-0.2cm}
\begin{enumerate}
\renewcommand{\labelenumi}{(\roman{enumi})}
\item $$c\big(v\cup w, L_1+L_2\big)\geq c(v,L_1)+c(w,L_2) $$
where $L_1+L_2$ is defined \footnote{Note that $L_1+L_2$ is not necessarily a smooth submanifold. However it is generated by the function $S_1\sharp S_2: E_1 \oplus E_2 \longrightarrow \mathbb{R}$ which is defined by $S_1\sharp S_2\,(x;\xi_1,\xi_2)=S_1(x;\xi_1)+S_2(x;\xi_2)$, where $S_1: E_1 \longrightarrow \mathbb{R}$ and
$S_2: E_2 \longrightarrow \mathbb{R}$ are generating functions for $L_1$ and $L_2$ respectively. For a cohomology class $u$ in $B$, by $c(u,L_1+L_2)$ we mean in fact $c(u,S_1\sharp S_2)$. If $L_1+L_2$ is a smooth submanifold contact isotopic to the 0-section, then this is consistent with the definition given above.} by
$$ L_1+L_2:=\{\;(q,p,z)\in J^1B \;|\; p=p_1+p_2, \; z=z_1+z_2, \;(q,p_1,z_1)\in L_1, \; (q,p_2,z_2)\in L_2  \;\}.$$
\item $$c(\mu,\bar{L})=-c(1,L),$$ where $\bar{L}$ denotes the image of $L$ under the map 
$J^1B\rightarrow J^1B$, $(q,p,z)\mapsto(q,-p,-z)$.
\item Assume that $L\cap0_B\neq\emptyset$. Then $c(\mu, L)=c(1,L)$ if and only if $L$ is the $0$-section. In this case we have $c(\mu,L)=c(1,L)=0$.
\end{enumerate}
\end{lemma}

The proof of this lemma is purely algebraic topological, and does not require any argument of symplectic or contact topology. It was originally given by Viterbo in the setting of Lagrangian submanifolds of $T^{\ast}B$ Hamiltonian isotopic to the 0-section, but its extension to the contact case is immediate (see \cite{B} or \cite{mio}). In the symplectic case, the symplectic character of $c$ is given by the fact that
\begin{equation}\label{V}
c\big(u,\Psi(L)\big)=c\big(u,L-\Psi^{-1}(0_B)\big)
\end{equation}
for every Hamiltonian symplectomorphism $\Psi$ of $T^{\ast}B$ (see \cite{V} or \cite{mio}). The analogue of this result does not hold in the contact case. However Bhupal proved that the following weaker statement is still true.

\begin{lemma}[\cite{B}]\label{weaker}
For any contactomorphism $\Psi$ of $J^1B$ contact isotopic to the identity, $0\neq u \in H^{\ast}(B)$ and $L\in \mathcal{L}$ it holds $$c\big(u,\Psi(L)\big)=0 \quad\Leftrightarrow \quad c\big(u,L-\Psi^{-1}(0_B)\big)=0.$$
\end{lemma}

The idea of the proof is to study the bifurcation diagram of a 1-parameter family $S_t$ of generating functions of $\Psi_t^{\phantom{t}-1}\Psi(L)-\Psi_t^{\phantom{t}-1}(0_B)$, where $\Psi_t$ is a contact isotopy connecting $\Psi$ to the identity, and to show that there cannot be a 1-parameter family $c_t$ of critical values of $S_t$ crossing the critical value $0$. The key reason why the critical value $0$ plays a special role is that critical points with critical value $0$ correspond to intersections of the generated Legendrian submanifold with the 0-section. As we observed in \cite{mio}, if $\Psi$ is 1-periodic in the $\mathbb{R}$-coordinate of $J^1B=T^{\ast}B\times\mathbb{R}$ then the argument of Bhupal can also be applied if we replace $0$ by any other integer, to get the following result. We denote by $\lceil\cdot\rceil$ (respectively $\lfloor\cdot\rfloor$) the smallest (respectively largest) integer that is greater or equal (respectively smaller or equal) to the given number.

\begin{lemma}[\cite{mio}]\label{1per}
Let $\Psi$ be a contactomorphism of $J^1B$ which is 1-periodic in the $\mathbb{R}$-coordinate of $J^1B=T^{\ast}B\times\mathbb{R}$, and isotopic to the identity through 1-periodic contactomorphisms. Then for every $u\neq 0$ in $H^{\ast}(B)$ and $L\in\mathcal{L}$ it holds
$$\lceil c\big(u,\Psi(L)\big)\rceil=\lceil c\big(u,L-\Psi^{-1}(0_B)\big)\rceil \quad \text{and} \quad \lfloor c\big(u,\Psi(L)\big)\rfloor=\lfloor c\big(u,L-\Psi^{-1}(0_B)\big)\rfloor.$$
\end{lemma}

\section{The invariants $c^+$ and $c^-$}\label{inv}

The invariants for Lagrangian submanifolds of $T^{\ast}B$ discussed in the previous section where applied by Viterbo \cite{V} to the special case of a compactly supported Hamiltonian symplectomorphism $\phi$ of $\big(\,\mathbb{R}^{2n}\,,\,\omega=dx\wedge dy\,\big)$, by regarding its compactified graph as a Lagrangian submanifold of $T^{\ast}S^{2n}$. Viterbo obtained in this way two invariants $c^+(\phi)$ and $c^-(\phi)$, defined by using respectively the orientation and unit cohomology classes of $S^{2n}$. The invariants $c^+$ and $c^-$ were generalized in \cite{B} and \cite{mio} respectively to the case of compactly supported contactomorphisms of $\mathbb{R}^{2n+1}$ and $\mathbb{R}^{2n}\times 
S^1$. We will review in this section the construction and properties of $c^+$ and $c^-$, discussing directly the contact case\footnote{Although $c^-$ did not appear in \cite{mio} it can be treated exactly as $c^+$, which is what in \cite{mio} we just called $c$.}.\\
\\
Let $\phi$ be a contactomorphism of $\big(\mathbb{R}^{2n+1},\xi_0=\text{ker}\,(dz-ydx)\big)$, with $\phi^{\ast}(dz-ydx)=e^g(dz-ydx)$. Following Bhupal \cite{B}, we define a Legendrian embedding $\Gamma_{\phi}:\mathbb{R}^{2n+1}\longrightarrow J^1\mathbb{R}^{2n+1}$ to be the composition $\tau \circ \text{gr}_{\phi}$, where $\text{gr}_{\phi}:\mathbb{R}^{2n+1}\longrightarrow \mathbb{R}^{2(2n+1)+1}$ is the Legendrian embedding $q\mapsto (q,\phi(q),g(q))$ and $\tau:\mathbb{R}^{2(2n+1)+1}\longrightarrow J^1\mathbb{R}^{2n+1}$ the contact embedding $(x,y,z,X,Y,Z,\theta)\mapsto \big(x,Y,z, Y-e^{\theta}y, x-X, e^{\theta}-1, xY-XY+Z-z\big)$. Here we consider the product contact structure $e^{\theta}(dz-ydx)-(dZ-YdX)$ on $\mathbb{R}^{2(2n+1)+1}$. More explicitely, for $\phi=(\phi_1,\phi_2,\phi_3)$ we have that $\Gamma_{\phi}:\mathbb{R}^{2n+1}\longrightarrow J^1\mathbb{R}^{2n+1}$ is given by
$$
\Gamma_{\phi}(x,y,z)=(x,\phi_2,z,\phi_2-e^gy,x-\phi_1,e^g-1,x\phi_2-\phi_1\phi_2+\phi_3-z).
$$
Note that $\Gamma_{\phi}$ can also be written as $\Gamma_{\phi}=\Psi_{\phi}\,(\text{0-section})$ where $\Psi_{\phi}$ is the local contactomorphism of $J^1\mathbb{R}^{2n+1}$ defined by the diagram
\begin{displaymath}
\xymatrix{
 \quad \mathbb{R}^{2(2n+1)+1} \quad \ar[r]^{\overline{\phi}} \ar[d]_{\tau} &
 \quad \mathbb{R}^{2(2n+1)+1} \quad \ar[d]^{\tau} \\
 \quad J^1\mathbb{R}^{2n+1} \quad \ar[r]_{\Psi_{\phi}} &  \quad J^1\mathbb{R}^{2n+1}}\quad
\end{displaymath}
with $\overline{\phi}$ the contactomorphism $(p,P,\theta)\mapsto (p,\phi(P), g(P)+\theta)$. This shows in particular that $\Gamma_{\phi}$ is contact isotopic to the 0-section. Notice also that the diagram above behaves well with respect to composition: for all contactomorphisms $\phi$ and $\psi$ we have namely that $\Psi_{\phi}\circ\Psi_{\psi}=\Psi_{\phi\psi} $ (in particular $\Gamma_{\phi\,\circ\,\psi}=\Psi_{\phi}\,(\Gamma_{\psi})$) and $\Psi_{\phi}^{\phantom{\phi}-1}=\Psi_{\phi^{-1}}$.\\
\\
If $\phi$ is compactly supported then we can see $\Gamma_{\phi}$ as a Legendrian submanifold of $J^1S^{2n+1}$, thus we can associate to it a generating function. The same is true if $\phi$ is a contactomorphism of $\mathbb{R}^{2n+1}$ which is 1-periodic in the $z$-coordinate and compactly supported in the $(x,y)$-plane (i.e. a compactly supported contactomorphism of $\mathbb{R}^{2n}\times S^1$) because then we can see $\Gamma_{\phi}$ as a Legendrian submanifold of $J^1\big(S^{2n}\times S^1\big)$. We denote respectively by $\text{Cont}_0^{\phantom{0}c}\,(\mathbb{R}^{2n+1})$ and $\text{Cont}_0^{\phantom{0}c}\,(\mathbb{R}^{2n}\times S^1)$ the groups of compactly supported contactomorphisms of $\mathbb{R}^{2n+1}$ and $\mathbb{R}^{2n}\times S^1$ that are isotopic to the identity. In the following we will always regard compactly supported contactomorphisms of $\mathbb{R}^{2n}\times S^1$ as 1-periodic contactomorphisms of $\mathbb{R}^{2n+1}$.\\
\\
For $\phi$ in $\text{Cont}_0^{\phantom{0}c}\,(\mathbb{R}^{2n+1})$ or in $\text{Cont}_0^{\phantom{0}c}\,(\mathbb{R}^{2n}\times S^1)$ we define 
$$c^+(\phi):=c(\mu,\Gamma_{\phi})$$
$$c^-(\phi):=c(1,\Gamma_{\phi})$$
where $\mu$ and $1$ are respectively the orientation and the unit cohomology class either of $S^{2n+1}$ or of $S^{2n}\times S^1$. Note that $c^+(\phi)$ and $c^-(\phi)$ are critical values for any generating function of $\Gamma_{\phi}$. Moreover, exactly as in the symplectic case, they satisfy the following properties.

\begin{lemma}\label{inv_sympl}
For all $\phi$, $\psi$ in $\text{Cont}_0^{\phantom{0}c}\,(\mathbb{R}^{2n+1})$ or $\text{Cont}_0^{\phantom{0}c}\,(\mathbb{R}^{2n}\times S^1)$ it holds:
\vspace{-0.2cm}
\begin{enumerate}
\renewcommand{\labelenumi}{(\roman{enumi})}
\item $c^+(\phi)\geq 0$ and $c^-(\phi)\leq 0$.
\item $c^+(\phi)=c^-(\phi)= 0$ if and only if $\phi$ is the identity.
\end{enumerate}
\end{lemma}

\begin{proof}
Point (ii) is an immediate consequence of Lemma \ref{inv_leg}(iii). As for (i), it can be seen as follows \cite{V}. We will prove that $c(1,\Gamma_{\phi})\leq 0$ and $c(1,\overline{\Gamma_{\phi}})\leq 0$ for any $\phi$, so that $c^-(\phi)\leq 0$ and, using Lemma \ref{inv_leg}(ii), $ c^+(\phi)=c(\mu,\Gamma_{\phi})=-c(1,\overline{\Gamma_{\phi}}) \geq 0$.
Since $c(1,\Gamma_{\phi})=\text{inf}\,\{\,a\in \mathbb{R} \;|\;i_{a}^{\phantom{a}\ast}(1)\neq 0\,\}$, we need to prove that
$i_0^{\phantom{0}\ast}(1)\neq 0$. Let $S: E\rightarrow\mathbb{R}$ be a g.f.q.i. for $\Gamma_{\phi}$ (respectively $\overline{\Gamma_{\phi}}$) and take a point $P$ in $B$, where $B$ denotes either $S^{2n+1}$ or $S^{2n}\times S^1$, outside the support of $\phi$. Consider the commutative diagram
\begin{displaymath}
\xymatrix{
 H^{\ast}(E^0,E^{-\infty})   \ar[r] &
 H^{\ast}(E_P^{\phantom{P}0},E_P^{\phantom{P}-\infty}) \\
 H^{\ast}(B) \ar[r] \ar[u]_{(i_0)^{\ast}} &
 H^{\ast}(\{P\}) \ar[u]_{\cong}}
\end{displaymath}
where the horizontal maps are induced by the inclusions $\{P\}\hookrightarrow B$ and $E_P\hookrightarrow E$. Since $P$ is outside the support of $\phi$ we have that $\Gamma_{\phi}$ and $\overline{\Gamma_{\phi}}$ coincide with the 0-section on a neighborhood of $P$, and so $S_{|E_P}: E_P \rightarrow \mathbb{R}$ is a quadratic form. It follows that the vertical map on the right hand side is an isomorphism. Since the horizontal map on the bottom sends 1 to 1, we see that $i_0^{\phantom{0}\ast}(1)\neq 0$ as we wanted.
\end{proof}

In the symplectic case the relation (\ref{V}), together with the properties in Lemma \ref{inv_leg}, implies that for every $\phi$, $\psi$ in $\text{Ham}^c\,(\mathbb{R}^{2n})$ we have $c^-(\phi)=-c^+(\phi^{-1})$, $c^+(\phi\psi)\leq c^+(\phi) + c^+(\psi)$ and $c^-(\phi\psi)\geq c^-(\phi) + c^-(\psi)$ (see \cite{V} or \cite{mio}). In the contact 1-periodic case we only get the following weaker statement, using Lemma \ref{1per}.

\begin{lemma}[\cite{mio}]\label{inv_per}
For all $\phi$, $\psi$ in $\text{Cont}_0^{\phantom{0}c}\,(\mathbb{R}^{2n}\times S^1)$ it holds:
\vspace{-0.2cm}
\begin{enumerate}
\renewcommand{\labelenumi}{(\roman{enumi})}
\item $\lfloor c^-(\phi)\rfloor=-\lceil c^+(\phi^{-1})\rceil$.
\item $\lceil c^+(\phi\psi)\rceil\leq\lceil c^+(\phi)\rceil+\lceil c^+(\psi)\rceil$ and $\lfloor c^-(\phi\psi)\rfloor\geq\lfloor c^-(\phi)\rfloor+\lfloor c^-(\psi)\rfloor$.
\end{enumerate}
\end{lemma}

\begin{proof}
\begin{enumerate}
\renewcommand{\labelenumi}{(\roman{enumi})}
\item Note first that $\lceil c(u,\Gamma_{\phi^{-1}})\rceil=\lceil c(u,\overline{\Gamma_{\phi}})\rceil$ for all $u$ (apply Lemma \ref{1per} to $L=0_B$ and $\Psi=\Psi_{\phi^{-1}}$). Using this and Lemma \ref{inv_leg}(ii) we have
$$
\lfloor c^-(\phi)\rfloor =\lfloor c(1,\Gamma_{\phi}) \rfloor=-\lceil c(\mu,\overline{\Gamma_{\phi}})\rceil=-\lceil c(\mu,\Gamma_{\phi^{-1}})\rceil=-\lceil c^+(\phi^{-1})\rceil.
$$
\item We have 
$ c^+(\psi)=c(\mu,\Gamma_{\psi})=c\big(\mu, \Psi_{\phi^{-1}}(\Gamma_{\phi\psi})\big)$ thus by Lemma \ref{1per} it holds
$\lceil c^+(\psi)\rceil=\lceil c\big(\mu,\Gamma_{\phi\psi}-\Psi_{\phi}(0_B)\big)\rceil$. But, by Lemma \ref{inv_leg}(i)-(ii)
$$
c\big(\mu,\Gamma_{\phi\psi}-\Psi_{\phi}(0_B)\big)\geq
c\big(\mu,\Gamma_{\phi\psi}\big)+c\big(1,\overline{\Gamma_{\phi}}\big)=
c^+(\phi\psi)-c^+(\phi).
$$
Thus 
$$\lceil c^+(\psi) \rceil \geq \lceil c^+(\phi\psi)-c^+(\phi) \rceil \geq \lceil c^+(\phi\psi)\rceil-\lceil c^+(\phi) \rceil$$
as we wanted. The statement about $c^-$ follows now from (i).
\end{enumerate}
\end{proof}

Similarly, in the case of $\mathbb{R}^{2n+1}$ we can use Lemma \ref{weaker} to show that $c^-(\phi)=0$ if and only if $c^+(\phi^{-1})=0$ and 
that if $c^{\pm}(\phi)=c^{\pm}(\psi)= 0$ then $c^{\pm}(\phi\psi)= 0$ (see \cite{B}).

A fundamental property of $c^+$ and $c^-$ in the symplectic case is that they are invariant by conjugation, i.e. $c^{\pm}(\phi)=c^{\pm}(\psi\phi\psi^{-1})$ for all $\phi$, $\psi$ in $\text{Ham}^c\,(\mathbb{R}^{2n})$ (see \cite{V} or \cite{mio}). This is a consequence of the fact that the set of critical values of the generating function of a Hamiltonian symplectomorphism $\phi$ of $\mathbb{R}^{2n}$ coincides with the action spectrum of $\phi$, which is invariant by conjugation: if $q$ is a fixed point of $\phi$ then $\psi(q)$ is a fixed point of $\psi\phi\psi^{-1}$ with the same symplectic action. This crucial fact does not hold in the contact case. Given a contactomorphism $\phi=(\phi_1,\phi_2,\phi_3)$ of $\mathbb{R}^{2n+1}$ with $\phi^{\ast}(dz-ydx)=e^g(dz-ydx)$, the critical points of a generating function of $\phi$ coincide with the \textit{translated points} of $\phi$, i.e. the points $q=(x,y,z)$ such that $\phi_1(q)=x$, $\phi_2(q)=y$ and $g(q)=0$. Moreover, the critical value is given by the \textit{contact action} of the corresponding translated point, i.e. the value $\phi_3(q)-z$ (see \cite{B} or \cite{mio}). Note that the contact action is not invariant by conjugation. In fact, not even the property of being a translated point is invariant by conjugation: if $q$ is a translated point of $\phi$ then in general $\psi(q)$ is not a translated point of $\psi\phi\psi^{-1}$. However this is the case if the contact action is $0$, because translated points with contact action $0$ are fixed points of $\phi$. This fact has been used by Bhupal to prove that, for all $\phi$, $\psi$ in $\text{Cont}_0^{\phantom{0}c}\,(\mathbb{R}^{2n+1})$, $c^{\pm}(\phi)=0$ if and only if $c^{\pm}(\psi\phi\psi^{-1})=0$. As for Lemma \ref{weaker}, the idea of the proof is to study the bifurcation diagram of a 1-parameter family $S_t$ of generating functions for $\psi_t\phi\psi_t^{\phantom{t}-1}$, where $\psi_t$ is a contact isotopy connecting $\psi$ to the identity, and to show that there can be no path $c_t$ of critical values for $S_t$ crossing the critical value $0$. As observed in \cite{mio}, in the 1-periodic case the same argument can also be applied if we replace $0$ by any other integer, to show that the integer part of $c^+$ and $c^-$ is invariant by conjugation.

\begin{lemma}[\cite{mio}]\label{conj_mio}
For all $\phi$, $\psi$ in $\text{Cont}_0^{\phantom{0}c}\,(\mathbb{R}^{2n}\times S^1)$ it holds that $\lceil c^{\pm}(\phi)\rceil=\lceil c^{\pm}(\psi\phi\psi^{-1})\rceil$ and $\lfloor c^{\pm}(\phi)\rfloor=\lfloor c^{\pm}(\psi\phi\psi^{-1})\rfloor$.
\end{lemma}

\begin{rmk}\label{lift}
Every Hamiltonian symplectomorphism $\varphi$ of $\mathbb{R}^{2n}$ can be lifted to a contactomorphism $\widetilde{\varphi}$ of $\mathbb{R}^{2n+1}$ or $\mathbb{R}^{2n}\times S^1$ by defining $\widetilde{\varphi}(x,y,z)=\big(\varphi_1(x,y),\varphi_2(x,y), z+F(x,y)\big)$ where $\varphi=(\varphi_1,\varphi_2)$ and $F$ is the compactly supported function satisfying $\varphi^{\ast}(ydx)-ydx=dF$. It can be proved (see \cite{mio}) that $c^+(\widetilde{\varphi})=c^+(\varphi)$ and $c^-(\widetilde{\varphi})=c^-(\varphi)$.
\end{rmk}

\section{The bi-invariant metric $d$ on $\text{Cont}_0^{\;c}\,(\mathbb{R}^{2n}\times S^1)$}\label{metric}

In \cite{V} Viterbo used the invariants $c^+$ and $c^-$ to construct a bi-invariant partial order $\leq_V$ and a bi-invariant metric $d_V$ on $\text{Ham}^c\,(\mathbb{R}^{2n})$, and a symplectic capacity for domains in $\mathbb{R}^{2n}$. Bhupal showed in \cite{B} that the weaker properties of $c^+$ and $c^-$ that are still satisfied in the case of $\text{Cont}_0^{\phantom{0}c}\,(\mathbb{R}^{2n+1})$ are in fact enough to extend the Viterbo partial order to that group. Note that, by Lemma \ref{inv_per} and Lemma \ref{conj_mio}, these same properties are also satisfied by elements of $\text{Cont}_0^{\phantom{0}c}\,(\mathbb{R}^{2n}\times S^1)$ so that Bhupal's contruction can be applied to the case of $\mathbb{R}^{2n}\times S^1$ as well (see \cite{mio}). However, we will now show that Lemma \ref{inv_per} and Lemma \ref{conj_mio}, that are only available in the 1-periodic case, allow us to extend also the Viterbo metric to $\text{Cont}_0^{\phantom{0}c}\,(\mathbb{R}^{2n}\times S^1)$.\\
\\
Recall that the Viterbo metric on $\text{Ham}^c\,(\mathbb{R}^{2n})$ is defined by $d_V\,(\phi,\psi):=c^+(\phi\psi^{-1})-c^-(\phi\psi^{-1})$. Similarly, our metric $d$ on $\text{Cont}_0^{\phantom{0}c} (\mathbb{R}^{2n}\times S^1)$ is defined by
$$d\,(\phi,\psi):=\lceil c^+(\phi\psi^{-1})\rceil-\lfloor c^-(\phi\psi^{-1})\rfloor.$$

\begin{prop}\label{metric_c}
$d$ is a bi-invariant metric on $\text{Cont}_0^{\phantom{0}c} (\mathbb{R}^{2n}\times S^1)$, i.e.
\begin{enumerate}
\renewcommand{\labelenumi}{(\roman{enumi})}
\item (positivity) $d\,(\phi,\psi)\geq 0$ for all $\phi$, $\psi$.
\item (non-degeneracy) $d\,(\phi,\psi)=0$ if and only if $\phi=\psi$.
\item (symmetry) $d\,(\phi,\psi)=d\,(\psi,\phi)$.
\item (triangle inequality) $d\,(\phi,\varphi)\leq d\,(\phi,\psi) + d\,(\psi,\varphi)$
\item (bi-invariance) $d\,(\phi\varphi,\psi\varphi)=d\,(\varphi\phi,\varphi\psi)=d\,(\phi,\psi)$.
\end{enumerate}
\end{prop}

\begin{proof}
Positivity and symmetry follow from Lemma \ref{inv_sympl}(i) and Lemma \ref{inv_per}(i) respectively. 
Using Lemma \ref{inv_per}(ii) we have
\begin{eqnarray*}
d(\phi,\varphi)&=&\lceil c^+(\phi\varphi^{-1})\rceil-\lfloor c^-(\phi\varphi^{-1})\rfloor=
\lceil c^+(\phi\psi^{-1}\psi\varphi^{-1})\rceil-\lfloor c^-(\phi\psi^{-1}\psi\varphi^{-1})\rfloor\\
&\leq & \lceil c^+(\phi\psi^{-1})\rceil+\lceil c^+(\psi\varphi^{-1})\rceil-\lfloor c^-(\phi\psi^{-1})\rfloor-\lfloor c^-(\psi\varphi^{-1})\rfloor=
d\,(\phi,\psi) + d\,(\psi,\varphi)
\end{eqnarray*}
proving the triangle inequality. By Lemma \ref{inv_sympl}(ii) we have $c^+(\text{id})=c^-(\text{id})=0$, thus $d\,(\phi,\phi)=0$. Suppose now that $d\,(\phi,\psi)=0$. Then, because of Lemma \ref{inv_sympl}(i), we must have $c^+(\phi\psi^{-1})=c^-(\phi\psi^{-1})=0$ and so $\phi=\psi$ by Lemma \ref{inv_sympl}(ii). This proves non-degeneracy. As for bi-invariance, we have
$$d\,(\phi\varphi,\psi\varphi)=
\lceil c^+(\phi\varphi\varphi^{-1}\psi^{-1})\rceil -\lfloor c^-(\phi\varphi\varphi^{-1}\psi^{-1})\rfloor=
\lceil c^+(\phi\psi^{-1})\rceil-\lfloor c^-(\phi\psi^{-1})\rfloor=d\,(\phi,\psi)$$
and, by Lemma \ref{conj_mio},
$$d\,(\varphi\phi,\varphi\psi)=
\lceil c^+(\varphi\phi\psi^{-1}\varphi^{-1})\rceil -\lfloor c^-(\varphi\phi\psi^{-1}\varphi^{-1})\rfloor=
\lceil c^+(\phi\psi^{-1})\rceil -\lfloor c^-(\phi\psi^{-1})\rfloor=d\,(\phi,\psi).$$
\end{proof}

The \textbf{energy} of an element $\phi$ of $\text{Cont}_0^{\phantom{0}c} (\mathbb{R}^{2n}\times S^1)$ is defined to be its distance to the identity, i.e.
$$E(\phi):=\lceil c^+(\phi)\rceil-\lfloor c^-(\phi)\rfloor.$$ 
Given an open and bounded domain $\mathcal{V}$ of 
$\mathbb{R}^{2n}\times S^1$, its \textbf{displacement energy} is defined by
$$ E(\mathcal{V}):=\text{inf}\;\{\; E(\psi) \;|\; \psi(\mathcal{V})\cap\mathcal{V}= \emptyset\;\}.$$
This definition can be extended to arbitrary domains of $\mathbb{R}^{2n}\times S^1$ by setting
$E(\mathcal{U})=\text{sup}\,\{\,E(\mathcal{V}) \;|\; \mathcal{V}\subset\mathcal{U}, \:\mathcal{V} \:\text{bounded}\,\}$
if $\mathcal{U}$ is open, and
$ E(A)=\text{inf}\,\{\,E(\mathcal{U}) \;|\; \mathcal{U} \:\text{open,}\;A\subset\mathcal{U}\,\}$
for an arbitrary domain $A$. In \cite{mio} we extended the Viterbo capacity to domains of $\mathbb{R}^{2n}\times S^1$ by defining 
$c(\mathcal{V})=\text{sup}\,\{\,\lceil c^+(\phi)\rceil \:|\: \phi\in\text{Cont}\,(\mathcal{V})\,\}$ 
where $\text{Cont}\,(\mathcal{V})$ denotes the set of time-1 maps of contact Hamiltonians supported in $\mathcal{V}$. The \textbf{energy-capacity inequality}
$$c(\mathcal{V})\leq E(\mathcal{V})$$
follows from \cite[3.6.1]{mio}.

\section{Relation with the Bhupal partial order}\label{order}

Recall from \cite{B} and \cite{mio} that, similarly to the symplectic case, the \textbf{Bhupal partial order} $\leq_B$ on $\text{Cont}_0^{\phantom{0}c}\,(\mathbb{R}^{2n+1})$ and $\text{Cont}_0^{\phantom{0}c}\,(\mathbb{R}^{2n}\times S^1)$ is defined by
$$ \phi_1 \leq_B \phi_2 \quad \text{if} \quad c^+(\phi_1\phi_2^{\phantom{2}-1})=0.$$
We will now show that the metric $d$ and the partial order $\leq_B$ are compatible.

\begin{prop}\label{c_part_ord_metric}
$\big(\,\text{Cont}_0^{\phantom{0}c} (\mathbb{R}^{2n}\times S^1),d,\leq_B\big)$ is a partially ordered metric space.
\end{prop}

\begin{proof}
Suppose that $\phi_1\leq_B\phi_2\leq_B\phi_3$. Then, since $c^+(\phi_1\phi_2^{\phantom{2}-1})=c^+(\phi_2\phi_3^{\phantom{3}-1})=c^+(\phi_1\phi_3^{\phantom{3}-1})=0$, using Lemma \ref{inv_per}(i) and (ii) we get
\begin{eqnarray*}
d(\phi_1,\phi_2)&=&\lceil c^+(\phi_1\phi_2^{\phantom{2}-1})\rceil-\lfloor c^-(\phi_1\phi_2^{\phantom{2}-1})\rfloor\\
&=&-\lfloor c^-(\phi_1\phi_2^{\phantom{2}-1})\rfloor=\lceil c^+(\phi_2\phi_1^{\phantom{1}-1})\rceil \leq \lceil c^+(\phi_2\phi_3^{\phantom{3}-1})\rceil+\lceil c^+(\phi_3\phi_1^{\phantom{1}-1})\rceil=\lceil c^+(\phi_3\phi_1^{\phantom{1}-1})\rceil\\
&=&
\lceil c^+(\phi_1\phi_3^{\phantom{3}-1})\rceil+\lceil c^+(\phi_3\phi_1^{\phantom{1}-1})\rceil=\lceil c^+(\phi_1\phi_3^{\phantom{3}-1})\rceil-\lfloor c^-(\phi_1\phi_3^{\phantom{3}-1})\rfloor=d(\phi_1,\phi_3)
\end{eqnarray*}
i.e. $d(\phi_1,\phi_2)\leq d(\phi_1,\phi_3)$.
\end{proof}

Consider now the relation $\preceq$ on $\text{Cont}_0^{\phantom{0}c} (\mathbb{R}^{2n}\times S^1)$ or $\text{Cont}_0^{\phantom{0}c} (\mathbb{R}^{2n+1})$ defined by setting $\phi_1\preceq\phi_2$ if $\phi_2\phi_1^{-1}$ can be written as the time-1 flow of a non-negative Hamiltonian. This relation is clearly reflexive and transitive. The deep fact that $\preceq$ is also anti-symmetric (hence a partial order) follows from antisymmetry of $\leq_B$ and the implication
\begin{equation}\label{e}
\phi_1\preceq\phi_2 \quad \Rightarrow \quad \phi_1\leq_B\phi_2.
\end{equation}
In the language of Eliashberg and Polterovich \cite{EP}, antisymmetry of $\preceq$ proves that $\mathbb{R}^{2n}\times S^1$ and  $\mathbb{R}^{2n+1}$ are orderable contact manifolds. The implication (\ref{e}) can easily be proved using the fact that $c^+$ and $c^-$ are monotone with respect to $\preceq$ : if $\phi_1\preceq\phi_2$ then $c^+(\phi_1)\leq c^+(\phi_2)$ and $c^-(\phi_1)\leq c^-(\phi_2)$ (see \cite{B} or \cite{mio}). Note that an analogous relation $\preceq$ is also defined in $\text{Ham}^c\,(\mathbb{R}^{2n})$ by setting $\phi_1\preceq\phi_2$ if $\phi_2\phi_1^{-1}$ can be written as the time-1 flow of a non-negative Hamiltonian. Then $\phi_1\leq_V\phi_2$ if 
$\phi_1\preceq\phi_2$, proving that $\preceq$ is a partial order (see \cite{V}).\\
\\
Notice that (\ref{e}) and Proposition \ref{c_part_ord_metric} immediately imply that $\big(\,\text{Cont}_0^{\phantom{0}c} (\mathbb{R}^{2n}\times S^1),d,\preceq\big)$ is a also partially ordered metric space, so that in particular the energy does not decrease along non-negative contact isotopies.

\section{Unboundedness of $d$}

We will now show that the diameter of $\text{Cont}_0^{\phantom{0}c} (\mathbb{R}^{2n}\times S^1)$ with respect to our metric is infinite, or in other words that $d$ is unbounded. As we will see, this fact follows immediately from the analogous result in the symplectic case and Remark \ref{lift}. \\
\\
Unboundedness of the Viterbo metric on $\text{Ham}^c\,(\mathbb{R}^{2n})$ is well known. It can be seen for instance by considering the sequence of Hamiltonian symplectomorphisms of $\mathbb{R}^{2n}$ supported in $B^{2n}(R)= \{\,\pi \sum_{i=1}^n x_i^{\,2} + y_i^{\,2} <R\,\}$ that was constructed by Traynor in \cite{T} to calculate the symplectic homology of $B^{2n}(R)$, and by noticing that the energy of the elements of this sequence tends to $R$ (that can be chosen to be arbitrarily big). Traynor's sequence $\phi^{\rho_1}\preceq\phi^{\rho_2}\preceq\phi^{\rho_3}\preceq\cdots$ is constructed as follows. Let $H:\mathbb{R}^{2n}\rightarrow\mathbb{R}$ be the function $H(x_1,y_1,\cdots,x_n,y_n)=\sum_{i=1}^{n} \frac{\pi}{R}(x_i^{\phantom{i}2}+y_i^{\phantom{i}2})$ and consider $H_{\rho}=\rho\circ H$, where $\rho:[0,\infty)\rightarrow [0,\infty)$ is a function supported in $[0,1]$ with $\rho''>0$. Take then a sequence $\rho_1$, $\rho_2$, $\rho_3$, $\cdots$ of functions of this form, with $\lim_{i \to \infty} \rho_i(0) = \infty$, $\lim_{i \to \infty} \rho'_i(0) = -\infty$, and such that
$H_{\rho_1}\leq H_{\rho_2}\leq H_{\rho_3}\leq\cdots$ with $H_{\rho_i}$ getting pointwise arbitrarily big on $B^{2n}(R)$. Since the $H_{\rho_i}$ are positive, by monotonicity of $c^-$ we have that $c^-(\phi^{\rho_i})=0$ and thus $E(\phi^{\rho_i})=c^+(\phi^{\rho_i})$. Moreover, it was proved by Traynor \cite{T} that $c^+(\phi^{\rho_i})$ tends to $R$ for $i\rightarrow \infty$. If we now lift the sequence $\phi^{\rho_1}\leq\phi^{\rho_2}\leq\phi^{\rho_3}\leq\cdots$ to $\mathbb{R}^{2n}\times S^1$ as explained in Remark \ref{lift}, we get a sequence of contactomorphisms whose energy tends to the integer part of $R$, which can be chosen arbitrarily big. It follows thus that $\text{Cont}_0^{\phantom{0}c} (\mathbb{R}^{2n}\times S^1)$ is also not bounded.\\
\\
The following terminology is taken from \cite{BIP}. Two norms on a group $G$ are said to be \textit{equivalent} if their ratio is bounded away from $0$ and $\infty$. In particular, a norm $ \nu$ on $G$ is equivalent to the trivial one (i.e. the norm that is everywhere $1$ except at the identity) if and only if it is bounded and not \textit{fine}. A norm $\nu$ on $G$ is called fine if $0$ is a limit point of $\nu(G)$. Since the metric $d$ on $\text{Cont}_0^{\phantom{0}c} (\mathbb{R}^{2n}\times S^1)$ takes values in $\mathbb{Z}$, it is not fine. However, being unbounded, it is not equivalent to the trivial one. An unbounded norm $\nu$ on a group $G$ is called \textit{stably unbounded} if $\lim_{n \to \infty} \frac{\nu(f^n)}{n}\neq 0$ for some $f$ in $G$. Note that, by definition of the capacity $c$, for every $\phi$ in $\text{Cont}_0^{\phantom{0}c} (\mathbb{R}^{2n}\times S^1)$ which is the time-1 flow of a Hamiltonian supported in $\mathcal{V}$ we have that
$E(\phi)=\lceil c^+(\phi)\rceil + \lceil c^+(\phi^{-1})\rceil\leq 2\,c(\mathcal{V})$. If $\phi$ is generated by a Hamiltonian supported in $\mathcal{V}$ then so is $\phi^n$ as well, thus $\lim_{n \to \infty} \frac{E(\phi^n)}{n}= 0$ for all $\phi$. Hence $\text{Cont}_0^{\phantom{0}c} (\mathbb{R}^{2n}\times S^1)$ is not stably unbounded. Similarly, $\text{Ham}^c\,(\mathbb{R}^{2n})$  is unbounded but not stably unbounded with respect to the Viterbo metric\footnote{The same also holds for $\text{Ham}^c\,(\mathbb{R}^{2n})$ with respect to the Hofer metric \cite{S3}.}.\\
\\
We conclude by discussing the case of the contact manifold $\big(S^1,\text{ker}(dz)\big)$, that can be seen as $\big(\mathbb{R}^{2n}\times S^1,\text{ker}(dz-ydx)\big)$ for $n=0$. It was proved in \cite{BIP} that the diffeomorphism group of $S^1$ does not admit any non-trivial (up to equivalence) bi-invariant metric. However, the construction of our metric $d$ on $\mathbb{R}^{2n}\times S^1$ does not contradict this result. First of all notice that, in the case $n=0$, $d$ is defined on the group of 1-periodic contactomorphisms of $\mathbb{R}$. While in higher dimension compactly supported contactomorphisms of $\mathbb{R}^{2n}\times S^1$ can be seen as 1-periodic contactomorphisms of $\mathbb{R}^{2n+1}$, there is no canonical way to do that if $n=0$ (we have namely that the group of 1-periodic contactomorphisms of $\mathbb{R}$ is the universal cover of the contactomorphism group of $S^1$). Moreover, as we will now explain, our construction does not give a metric even on the group of 1-periodic contactomorphisms of $\mathbb{R}$. Note that for a diffeomorphism $\phi$ of $\mathbb{R}$ it holds that $\phi^{\ast}dz=\phi'dz$, thus $\phi$ is a contactomorphism if and only if $\phi'>0$ i.e. if and only if $\phi$ is orientation preserving. Let $\phi$ be an orientation preserving 1-periodic diffeomorphism of $\mathbb{R}$. We can then associate to $\phi$ a Legendrian submanifold $\Gamma_{\phi}$ of $J^1\mathbb{R}$ by defining $\Gamma_{\phi}=\tau\circ\text{gr}_{\phi}$, where $\tau:\mathbb{R}^3\rightarrow J^1\mathbb{R}$ is the contact embedding 
$$\tau(z,Z,\theta)=\big(\,z,e^{\theta}-1,Z-z\,\big).$$
Thus $\Gamma_{\phi}:\mathbb{R}\rightarrow J^1\mathbb{R}$ is given by $\Gamma_{\phi}(q)=\big(\,q,\phi'(q)-1,\phi(q)-q\,\big)$. We have that $\Gamma_{\phi}$ is the 1-jet of the function $S:\mathbb{R}\rightarrow\mathbb{R}$, $S(q)=\phi(q)-q$ thus in other words $S$ is a generating function for $\phi$. Notice that $c^+(\phi)=\text{max}(S)$ is not necessarily non-negative, and $c^-(\phi)=\text{min}(S)$ not necessarily non-positive. Hence if we define $d$ as in Section \ref{metric} we do not get a metric in this case, because positivity fails. Notice that the proof of non-negativity of $c^+(\phi)$ and non-positivity of $c^-(\phi)$ in Lemma \ref{inv_sympl}(i) used in a crucial way the possibility of choosing a point outside the support of $\phi$. This cannot be done in general in the $S^1$ case.

\end{document}